\documentclass[12pt]{article}

\usepackage{mathrsfs}
\usepackage{amsmath}
\usepackage{amsthm}
\usepackage{amssymb}
\usepackage{amsfonts}
\usepackage{epsf}
\usepackage{graphicx}
\usepackage{hyperref}
\usepackage{comment}
\usepackage{overpic}

\newcommand{\D}{\displaystyle}
\newcommand{\rex}{{\rm e}}
\newtheorem{thm}{Theorem}[section]
\newtheorem{lem}[thm]{Lemma}

\newtheorem{prop}[thm]{Proposition}
\theoremstyle{definition}

\theoremstyle{remark}
\newtheorem{rem}[thm]{Remark}

\numberwithin{equation}{section}

\newcommand{\al}{\alpha}
\newcommand{\bt}{\beta}

\begin{document}

\title{Painlev\'e V and a Pollaczek-Jacobi type orthogonal polynomials}
\date{}
\author{Yang Chen \\
Department of Mathematics, Imperial College London, \\
180 Queens Gates, London SW7 2BZ, UK. \\
ychen@ic.ac.uk \\
and \\
Dan Dai \\
Department of Mathematics, City University of Hong Kong, \\
Tat Chee Avenue, Kowloon, Hong Kong.\\
and \\
Department of Mathematics, Katholieke Universiteit Leuven, \\
Celestijnenlaan 200 B, 3001 Leuven, Belgium. \\
dandai@cityu.edu.hk} \maketitle

%%%%%%%%%%%%%%%%%%%%%%%%%%%%%%%%%%%%%%%%%%%%%%%%%%%%%%%%%%%%

\begin{abstract}
We study a sequence of polynomials orthogonal with respect to a one
parameter family of weights
$$
w(x):=w(x,t)=\rex^{-t/x}\:x^{\al}(1-x)^{\bt},\quad t\geq 0,
$$
defined for $x\in[0,1].$ If $t=0,$ this reduces to a shifted Jacobi
weight. Our ladder operator formalism and the associated
compatibility conditions give an easy determination of the
recurrence coefficients.

For $t>0,$ the factor $\rex^{-t/x}$ induces an infinitely strong
zero at $x=0.$ With the aid of the compatibility conditions, the
recurrence coefficients are expressed in terms of a set of auxiliary
quantities that satisfy a system of difference equations. These,
when suitably combined with a pair of Toda-like equations derived
from the orthogonality principle, show that the auxiliary quantities
are a particular Painlev\'e V and/or allied functions.

It is also shown that the logarithmic derivative of the Hankel
determinant,
$$
D_n(t):=\det\left(\int_{0}^{1} x^{i+j}
\:\rex^{-t/x}\:x^{\al}(1-x)^{\bt}dx\right)_{i,j=0}^{n-1},
$$
satisfies the Jimbo-Miwa-Okamoto $\sigma-$form of the Painlev\'e V
and that the same quantity satisfies a second order non-linear
difference equation which we believe to be new.
\end{abstract}

%%%%%%%%%%%%%%%%%%%%%%%%%%%%%%%%%%%%%%%%%%%%%%%%%%%%%%%%%%%%%%%%%%%%%%%%

\section{Introduction}

For polynomials orthogonal with respect to weights $w$ absolutely
continuous on $[-1,1]$ and
 satisfies the Szeg\H{o} condition,
$$
\int_{-1}^{1}\frac{|\ln w(x)|}{\sqrt{1-x^2}}dx<\infty,
$$
a general theory of Szeg\H{o} \cite[296-312]{s} gives a
comprehensive description of the large $n$ behavior of the
polynomials both for $x\in(-1,1)$ and $x\notin[-1,1]$ and the
recurrence coefficients; see also \cite{Ge}. For a recent account
and extension of the Szeg\H{o}'s theory, see \cite{Ne} and
\cite{Den,Ne1}.

With the introduction of a deformation parameter $t\geq 0,$ we have
a Pollaczek-Jacobi type weight defined as
\begin{equation} \label{w-def}
w(x):= w(x;t) = \rex^{-t/x} x^\alpha (1-x)^\beta, \qquad x \in
[0,1], \;\; \alpha>0, \;\;\beta>0,
\end{equation}
which violates the Szeg\H{o} condition. For convenience we have
taken the interval of orthogonality to be $[0,1].$ The weight
function of the Pollaczek polynomial and a generalization due to
Szeg\H{o} behaves as
$$
\exp\left(-\frac{c}{\sqrt{1-x^2}}\right),{\;\;\rm as\;}\;\;x\to \pm
1,
$$
where $c$ is a positive constant; see \cite[393-400]{s} for a
detailed description. We should like to mention, as a brief guide to
the reader, some recent literature on the Pollaczek polynomials; see
\cite{van} for the asymptotic behavior of the polynomials for
$x\notin[-1,1]$; \cite{Bo, Zhou} for the asymptotic behavior of
their zeros and \cite{Fink,mun,Wat} for applications to physical
problems. Regarding weighted polynomial approximation in $L_p$ with
respect to a class of exponential weights on $[-1,1]$ that violate
the Szeg\"o conditions, see \cite{dem,lu}.

Note that our weight is in some sense more ``singular" since the
Szeg\H{o} condition is strongly violated.

The purpose of this paper is to give a complete description of the
recurrence coefficients of the associated orthogonal polynomials. As
can be seen later these are expressed in terms of a set auxiliary
quantities which are ultimately a particular Painlev\'e V and allied
functions.

In section 2, with the aid of certain supplementary conditions
$(S_1),$ $(S_2)$ and $(S_2')$ derived from a pair of operators and
the
 recurrence relations, we obtain a system of difference
equations satisfied by certain auxiliary quantities
($R_n,\;R_n^*,\;r_n,\;r_n^*$) and the recurrence coefficients
($\al_n,\;\bt_n).$ More importantly, the equation $(S_2'),$ in some
sense is the ``first integral" of $(S_1)$ and $(S_2),$ and
automatically performs a sum, in finite term, of $R_j^*$ from $j=0$
to $j=n-1$ (see (\ref{Rsum})). This turns out to be the logarithmic
derivative of the Hankel determinant (generated by our
Pollaczek-Jacobi type weight) with respect to $t$; see
(\ref{dn-R*}).

We should mention here similar approach was adopted in
\cite{basor-chen,ci} where $(S_1)$ and $(S_2')$ were sufficient for
the purpose. However, for the problem at hand, taking into account
of the difference equations there are ultimately three auxiliary
quantities $R_n,\;r_n^*$ and $r_n$. The equation $(S_2)$ turns out
to be crucial for later development.

In section 3, we make use of the results of section 2 to express
$\al_n$ and $\bt_n,$ in terms of the auxiliary quantities and show
that these reduce to recurrence coefficients of the ``shifted"
Jacobi polynomials when $t=0$ through an easy computation.

The $t$ dependence of the recurrence coefficients and the auxiliary
quantities is derived in section 4 resulting a pair of Toda-like
equations.

In section 5, combining the results from previous sections we
express $r_n^*,\;r_n$ and $R_n$ in term of
\begin{equation} \label{hn-def}
H_n(t):=t\frac{d}{dt}\ln D_n(t)
\end{equation}
and $H_n'(t).$ We show that a functional equation $f(H_n,
H_n',H_n'')=0$ resulting from eliminating the auxiliary quantities
in favor of $H_n$ and $H_n'$ is the Jimo-Miwa-Okamoto $\sigma-$form
of a Painlev\'e V equation.

In section 6, from the difference equations found in section 2 and
the expression of $\al_n$ and $\bt_n$ found in section 3 we express
$R_n,$ $\bt_n$ in terms of $\textsf{p}_1(n),$ the coefficients of
$z^{n-1}$ of our monic polynomials $P_n(z)$. And since
$\textsf{p}_1(n)$ is easily related to $H_n,$ the resulting
functional equation $g(H_{n},H_{n+1},H_{n-1})=0$ is the discrete
analog of the $\sigma-$form mentioned in the abstract; see
(\ref{final-difference}). We believe this equation is new.

In section 7 we derive a second order o.d.e. satisfied by $R_n$
which is also a Painlev\'e V since $H_n$
 is shown to satisfies the $\sigma-$form.

The large $n$ behavior of the recurrence coefficients will be
described in a future publication.

%%%%%%%%%%%%%%%%%%%%%%%%%%%%%%%%%%%%%%%%%%%%%%%%%%%%%%%%%%%%%%%%%%%%%%%%

\section{Preliminaries}

Let $P_n(x)$ be the monic polynomials of degree $n$ in $x$ and
orthogonal with respect to the weight function $w(x;t)$ defined in
(\ref{w-def}), that is
\begin{equation} \label{orthogonality}
\int_0^1 P_m(x) P_n(x) w(x;t) dx = h_n \delta_{m,n}.
\end{equation}
(The polynomials $P_n(x)$ and the constant $h_n$ all depend on $t$,
but we suppress the dependence for brevity.) An immediate
consequence of the orthogonality condition is recurrence relation
\begin{equation} \label{recurrence}
x P_n(x) = P_{n+1}(x) + \alpha_n P_n(x) + \beta_n P_{n-1}(x), \qquad
n=0,1,\cdots.
\end{equation}
We take the ``initial" conditions to be $P_0(z) := 1$ and $\beta_0
P_{-1}(z) :=0$. Then it is easily seen that $P_n(z)$ has the
following form
\begin{equation} \label{pn-formula}
P_n(z) = z^n + \textsf{p}_1(n) z^{n-1} + \cdots.
\end{equation}
Substituting (\ref{pn-formula}) into (\ref{recurrence}) we see that
\begin{equation} \label{alpha-p1n}
\alpha_n = \textsf{p}_1(n) - \textsf{p}_1(n+1).
\end{equation}
Taking a telescopic sum of above equation and noting that
$\textsf{p}_1(0):=0,$ we have
\begin{equation} \label{alphasum-p1n}
-\sum_{j=0}^{n-1}\al_j=\textsf{p}_1(n).
\end{equation}
The Hankel determinant generated by our weight is
\begin{align} \label{dn-def}
D_n(t):=&\det\left(\mu_{j+k}(t)\right)_{j,k=0}^{n-1}\nonumber\\
=&\prod_{j=0}^{n-1} h_j,
\end{align}
where
\begin{align}
\mu_{k}(t):=&\int_{0}^{1}x^{k} e^{-t/x}x^{\al}(1-x)^{\beta}dx\nonumber\\
=& \: e^{-t}\:\Gamma(1+\bt)U(1+\bt,-\al-k,t)\nonumber
\end{align}
and $U(a,b,z)$ the Kummer function of the second kind; see
\cite{slater}.

The Hankel determinant will turn out to play an important role in
our determination of $\alpha_n$ and $\beta_n$ for the weight given
by (\ref{w-def}). Furthermore, through several auxiliary variables
(which naturally appear in the theory) we obtain expressions of
$\al_n$ and $\bt_n$ terms of $H_n(t)$ given in (\ref{hn-def}) and
its derivatives with respect to $t$. The equations
(\ref{orthogonality})--(\ref{dn-def}) can be found in Szeg\H{o}'s
treatise \cite{s} on orthogonal polynomials.

If $w(x)$, the weight function, is Lipschitz continuous then the
following ladder operator relations hold
\begin{align}
\left( \frac{d}{dz} + B_n(z) \right) P_n(z) & = \beta_n A_n(z) P_{n-1}(z), \label{ladder1} \\
\left( \frac{d}{dz} - B_n(z) - \textsf{v}'(z) \right) P_{n-1}(z) & =
- A_{n-1}(z) P_n(z) \label{ladder2}
\end{align}
with
\begin{align}
A_n(z) & := \frac{1}{h_n} \int_0^1 \frac{\textsf{v}'(z) -
\textsf{v}'(y)}{z-y} \ [P_n(y)]^2 w(y) dy,
\label{an-def}\\
B_n(z) & := \frac{1}{h_{n-1}} \int_0^1 \frac{\textsf{v}'(z) -
\textsf{v}'(y)}{z-y} \ P_{n-1}(y) P_n(y) w(y) dy \label{bn-def}
\end{align}
and $\textsf{v}(z):=-\ln w(z).$

Note that, as a consequence of the recurrence relation and the
Christoffel-Darboux formula, $A_n(z)$ and $B_n(z)$ are not
independent but must satisfy the following supplementary conditions
valid for $z\in\mathbb{C}\cup \{ \infty \}.$
\begin{thm}
The functions $A_n(z)$ and $B_n(z)$ defined by (\ref{an-def}) and
(\ref{bn-def}) satisfy the identities,
$$
B_{n+1}(z) + B_n(z)  = (z- \alpha_n) A_n(z) - \textsf{v}'(z),
\eqno(S_1)
$$
$$
1+ (z- \al_n) [B_{n+1}(z) - B_n(z)] = \beta_{n+1} A_{n+1}(z) -
\beta_n A_{n-1}(z).\eqno(S_2)
$$
\end{thm}
\begin{proof}
The ladder operator relations (\ref{ladder1}) and (\ref{ladder2})
and the supplementary conditions $(S_1)$ and $(S_2)$ have been
derived by many authors in different forms over the years
\cite{BonanC,BonanLN,Bonan,sh}. Also, see \cite{ci1,ci2} for a
recent proof.
\end{proof}

It turns out that there is another identity involving
$\sum_{j=0}^{n-1}A_j$ which will provide further insight into the
determination of $\al_n$ and $\bt_n.$ We state the result in the
following Theorem.

\begin{thm}

$$
B_n^2(z) + \textsf{v}'(z) B_n(z) + \sum_{j=0}^{n-1}A_j(z) =
\beta_{n} A_n(z) A_{n-1}(z). \eqno(S_2')
$$

\end{thm}

\begin{proof}

This can be obtained as follows: First, we multiply $(S_2)$ by
$A_n(z)$ and replace $(z-\al_n)A_n(z)$ in the resulting equation by
$B_{n+1}(z)+B_n(z)+\textsf{v}'(z)$ with $(S_1)$ to get
$$
B_{n+1}^2(z) - B_n^2(z) + \textsf{v}'(z) (B_{n+1}(z) - B_n(z))+
A_n(z) = \beta_{n+1} A_{n+1}(z) A_n(z) - \beta_n A_n(z) A_{n-1}(z) .
$$
Taking a sum of the above equation from 0 to $n-1$, we obtain our
Theorem with the initial conditions $B_0(z)=0$ and
$\bt_{0}A_{-1}(z)=0.$

Let us explain a little more about the way we get the above initial
conditions. In (\ref{bn-def}) we re-write
$$
\frac{P_{n-1}(y)}{h_{n-1}},
$$
as
$$
\frac{\bt_n}{h_n}P_{n-1}(y).
$$
Consequently $B_0(z)=0,$ since $\bt_0P_{-1}(y)=0$ according to the
initial condition associated with the recurrence relations
(\ref{recurrence}). $\bt_0 A_{-1}(z)=0$ for the same reason.
\end{proof}

For the problem at hand,
\begin{equation}
\textsf{v}(z):= - \ln w(z) = \frac{t}{z} - \alpha \ln{z} - \beta
\ln(1-z).
\end{equation}
Then we immediately have
\begin{equation} \label{vz'}
\textsf{v}'(z) = -  \frac{t}{z^2} - \frac{\alpha}{z} -
\frac{\beta}{z-1}
\end{equation}
and
\begin{equation} \label{vz'2}
\frac{\textsf{v}'(z) - \textsf{v}'(y)}{z-y} = \frac{t}{z^2 y} +
\frac{\alpha y + t}{z y^2} + \frac{\beta}{(z-1)(y-1)}.
\end{equation}
Substituting the above formula into the definitions of $A_n(z)$ and
$B_n(z)$ in (\ref{an-def}) and (\ref{bn-def}), we have the following
proposition.

\begin{prop}
We have
\begin{align}
A_n(z) & = \frac{R^*_n}{z^2} + \frac{R_n}{z} - \frac{R_n}{z-1}, \label{an-new}\\
B_n(z) & = \frac{r^*_n}{z^2} - \frac{n - r_n}{z} - \frac{r_n}{z-1},
\label{bn-new}
\end{align}
where
\begin{align}
R^*_n & := \frac{t}{h_n} \int_0^1 [P_n(y)]^2 w(y) \frac{dy}{y}, \label{rn1-def}\\
R_n & := \frac{\beta}{h_n} \int_0^1 [P_n(y)]^2 w(y) \frac{dy}{1-y}, \label{r*n1-def} \\
r^*_n & := \frac{t}{h_{n-1}} \int_0^1 P_{n-1}(y) P_n(y) w(y) \frac{dy}{y}, \label{rn2-def} \\
r_n & := \frac{\beta}{h_{n-1}} \int_0^1 P_{n-1}(y) P_n(y) w(y)
\frac{dy}{1-y}. \label{r*n2-def}
\end{align}

\end{prop}

\begin{proof}
Using (\ref{vz'2}), (\ref{an-def}) can be rewritten as
\begin{equation} \label{an-1}
\begin{split}
A_n(z) = & \frac{1}{h_n} \left[ \frac{1}{z^2} \int_0^1 [P_n(y)]^2
w(y) \frac{t}{y} dy + \frac{1}{z} \int_0^1
[P_n(y)]^2 w(y) \frac{\alpha y + t}{y^2}  dy \right. \\
& \left. \quad + \frac{1}{z-1} \int_0^1 [P_n(y)]^2 w(y)
\frac{\beta}{y-1} dy \right].
\end{split}
\end{equation}
Applying integration by parts, we have
\begin{equation}
\int_0^1 [P_n(y)]^2 w(y) \textsf{v}'(y) dy = - \int_0^1 [P_n(y)]^2 d
w(y) = \int_0^1 2 P_n'(y) P_n(y) w(y) d y = 0.
\end{equation}
Then it follows from (\ref{vz'}) and the above formula that
\begin{equation} \label{int-p}
\int_0^1 [P_n(y)]^2 w(y) \frac{\alpha y + t}{y^2} dy = - \int_0^1
[P_n(y)]^2 w(y) \frac{\beta}{y-1} dy.
\end{equation}
Combining (\ref{an-1}) and (\ref{int-p}) gives us (\ref{an-new}).

In a very similar way, we get (\ref{bn-new}) from (\ref{bn-def}).
One just needs to take into account the following equality
\begin{equation}
\int_0^1 P_{n-1}(y) P_n(y) w(y) \frac{\alpha y + t}{y^2} dy = -n
h_{n-1} - \int_0^1 P_{n-1}(y) P_n(y) w(y) \frac{\beta}{y-1} dy.
\end{equation}
\end{proof}

Now we have four more auxiliary quantities $R_n, R_n^*, r_n, r_n^*$,
in addition to the two unknowns $\al_n$ and $\bt_n$. However, from
($S_1$), ($S_2$) and ($S_2'$), we obtain relations among these
quantities.

\begin{prop} \label{eqns-r&R}

From ($S_1$), we obtain the following equations
\begin{align}
r^*_{n+1} + r^*_n & = t - \alpha_n R^*_n, \label{r-r1} \\
R^*_n - R_n &  = -2n - 1 - \alpha - \beta, \label{r-r2} \\
r_{n+1} + r_n  & = (1- \alpha_n)R_n - \beta , \label{r-r3}
\end{align}
where the constants $R_n$, $R_n^*$, $r_n$ and $r^*_n$ are defined in
(\ref{rn1-def})--(\ref{r*n2-def}), respectively.

\end{prop}

\begin{proof}
Substituting (\ref{an-new}) and (\ref{bn-new}) into $(S_1)$, we get
\begin{equation}
B_{n+1}(z) + B_n(z)  = \frac{r^*_{n+1} + r^*_n}{z^2} - \frac{-
r_{n+1} - r_n + 2n+1}{z} - \frac{r_{n+1} + r_n}{z-1}
\end{equation}
and
\begin{align}
(z-\alpha_n) A_n(z) - \textsf{v}'(z) & = (z-\alpha_n) \left[
\frac{R^*_n}{z^2} + \frac{R_n}{z} - \frac{R_n}{z-1} \right] - \left[
-  \frac{t}{z^2} - \frac{\alpha}{z} + \frac{\beta}{1-z} \right] \nonumber \\
& = \frac{t - \alpha_n R^*_n }{z^2} + \frac{\alpha + R^*_n -
\alpha_n R_n}{z} + \frac{\beta - (1-\alpha_n)R_n}{z-1}.
\end{align}
Comparing the coefficients in the above two formulas, it follows
that
\begin{align}
r^*_{n+1} + r^*_n & = t - \alpha_n R^*_n, \\
- r_{n+1} - r_n + 2n+1 & = - \alpha - R^*_n + \alpha_n R_n, \\
- r_{n+1} - r_n & = \beta - (1-\alpha_n)R_n.
\end{align}
Combining the above three formulas immediately proves our
proposition.
\end{proof}

\begin{prop}

From ($S_2'$), we obtain the following equations
\begin{equation} \label{r&R}
(r^*_n)^2 - t r^*_n = \beta_n R^*_n R^*_{n-1},
\end{equation}
\begin{equation} \label{r*&R*}
r_n^2 + \beta r_n = \beta_n  R_n R_{n-1},
\end{equation}
\begin{equation} \label{r&R&R*}
(t - 2 r^*_n) (n - r_n)- \alpha r^*_n = \beta_n(R^*_n R_{n-1} +
R^*_{n-1} R_n)
\end{equation}
and
\begin{equation} \label{Rsum}
\sum_{j=0}^{n-1} R^*_j = n (t - \alpha - n ) - ( 2n + \alpha + \beta
) (r^*_n - r_n ),
\end{equation}
where the constants $R_n$, $R_n^*$, $r_n$ and $r^*_n$ are defined in
(\ref{rn1-def})--(\ref{r*n2-def}), respectively.
\end{prop}

\begin{proof}

From (\ref{an-new}) and (\ref{bn-new}), we know
\begin{equation} \label{s2'lhs}
\begin{split}
B_n^2(z) + \textsf{v}'(z) B_n(z) + \sum_{j=0}^{n-1}A_j(z) &  \\
& \hspace{-130pt} = \frac{(r^*_n)^2 - t r^*_n}{z^4} + \frac{(t - 2
r^*_n) (n - r_n)- \alpha r^*_n}{z^3}
+ \frac{(n - r_n)^2 + \alpha (n - r_n)}{z^2}  \\
& \hspace{-130pt} \quad + \frac{r_n^2 + \beta r_n}{(z-1)^2}+ \frac{-
2 r_n r^*_n - \beta r^*_n + t r_n}{z^2 (z-1)} + \frac{( \beta + 2
r_n) (n - r_n) + \alpha
r_n}{z (z-1)} \\
&  \hspace{-130pt} \quad + \sum_{j=0}^{n-1} \left[ \frac{R^*_j}{z^2}
+ \frac{R_j}{z} - \frac{R_j}{z-1} \right].
\end{split}
\end{equation}
Using (\ref{an-new}) again, we have
\begin{align}
\beta_{n} A_n(z) A_{n-1}(z)  & =  \frac{\beta_n R^*_n R^*_{n-1}
}{z^4} +  \beta_n R_n R_{n-1} \left[
\frac{1}{z^2} - \frac{2}{z(z-1)} + \frac{1}{(z-1)^2} \right] \nonumber \\
& \quad + \beta_n(R^*_n R_{n-1} + R^*_{n-1} R_n) \left[
\frac{1}{z^3} - \frac{1}{z^2(z-1)} \right] .  \label{s2'rhs}
\end{align}
Recalling $(S_2')$, (\ref{s2'lhs}) and (\ref{s2'rhs}) are equal.
Then let us compare their coefficients. At $O(z^{-4})$,
$O(z-1)^{-2}$ and $O(z^{-3})$, equating the coefficients we have
(\ref{r&R}), (\ref{r*&R*}) and (\ref{r&R&R*}) in our proposition,
respectively. At $O(z^{-2})$, using the fact
\begin{equation*}
\frac{1}{z-1} = - 1 - z - z^2 - \cdots, \qquad \textrm{as }  z \to
0,
\end{equation*}
we obtain
\begin{equation}
\begin{split}
(n - r_n)^2 + \alpha (n - r_n) - (- 2 r_n r^*_n - \beta r^*_n + t r_n) = & \ \beta_n R_n R_{n-1} \\
& \hspace{-180pt} + \beta_n(R^*_n R_{n-1} + R^*_{n-1} R_n) -
\sum_{j=0}^{n-1} R^*_j.
\end{split}
\end{equation}
Combining (\ref{r&R&R*}) and the above formula yields
\begin{equation} \label{Rsum0}
\begin{split}
\sum_{j=0}^{n-1} R^*_j = & \ \beta_n  R_n R_{n-1} + (t - 2 r^*_n - \alpha)(n - r_n) \\
&  - (n - r_n)^2 + (t - 2 r^*_n) r_n - (\alpha + \beta) r^*_n.
\end{split}
\end{equation}
Substituting (\ref{r*&R*}) into (\ref{Rsum0}) gives us (\ref{Rsum}).
\end{proof}

\begin{rem}
From ($S_2$), using similar calculations as in the above
proposition, we get one more equation as follows
\begin{equation} \label{r&r4}
- r_{n+1} + r_n + r^*_{n+1} - r^*_n + \alpha_n  = 0.
\end{equation}
To continue, we re-write (\ref{r&r4}) as,
\begin{equation}
-\al_n=r_{n+1}^*-r_n^*-(r_{n+1}-r_n).\nonumber
\end{equation}
Performing a telescopic sum and recalling (\ref{alphasum-p1n}), we
find the very handy relation
\begin{equation} \label{p1n-rs}
\textsf{p}_1(n)=r_n^*-r_n,
\end{equation}
where we have used the initial conditions $r_0(t)=r_0^*(t):=0.$ As
we shall see later (\ref{p1n-rs}) will play a crucial role in the
derivation of the Painlev\'e equation.
\end{rem}
\begin{rem}
We may expect that $(S_1)$ and $(S_2)$ should ``contain" all that is
necessary. However, these are non-linear equations and their
combination $(S_2')$ carries extra information. It transpires that
all three are needed to provide a completely description of the
recurrence coefficients.
\end{rem}
%%%%%%%%%%%%%%%%%%%%%%%%%%%%%%%%%%%%%%%%%%%%%%%%%%%%%%%%%%%%%%%%5

\section{The recurrence coefficients}

In this section we shall express the recurrence coefficients $\al_n$
and $\bt_n$ in terms of the auxiliary quantities $R_n,\:r_n$ and
$r_n^*.$ Note that we do not require $R_n^*$ since it is $R_n$ up to
a linear form in $n$; see (\ref{r-r2}).

\begin{lem}
The diagonal recurrence coefficients $\al_n$ is expressed in terms
of $R_n,\:r_n,\:$ and $r_n^*$ as follows:
\begin{equation} \label{alpha-rs}
(2n+2+\al+\bt)\al_n=2(r_n^*-r_n)+R_n-\bt-t.
\end{equation}
\end{lem}

\begin{proof}

We eliminate $R_n^*$ from (\ref{r-r1}) with the aid of (\ref{r-r2})
and find
\begin{equation}
r_{n+1}^*+r_n^*=t+\al_n(2n+1+\al+\bt-R_n).
\end{equation}
Subtracting the above formula from (\ref{r-r3}), we get
\begin{equation}
r_{n+1}-r_{n+1}^*+r_n-r_n^*=R_n-\bt-t-(2n+1+\al+\bt)\al_n.
\end{equation}
Recalling (\ref{r&r4}), we see that the left hand side of the above
formula is $\al_n + 2 (r_n-r_n^*)$. Then (\ref{alpha-rs})
immediately follows.
\end{proof}

\begin{rem}
For $n=0,$ we find, from the definition of $\al_0(t)$ and $R_0(t),$
that
\begin{align}
\al_0(t)=&\frac{U(1+\bt,-\al-1,t)}{U(1+\bt,-\al,t)} ,\\
R_0(t)=&\frac{U(\bt,-\al,t)}{U(1+\bt,-\al,t)},
\end{align}
where $U$ is the second solution of Kummer's equation; see
\cite{slater}. We verify the validity of (\ref{alpha-rs}) at $n=0$
by substituting the above two formulas.
\end{rem}
\begin{rem} \label{R0-asy}
For $t\to\infty,$
\begin{equation}
R_0(t)=t\left(1+\frac{\al+2(1+\bt)}{t}+{\rm
O}\left(1/t^2\right)\right).
\end{equation}
\end{rem}
The next lemma gives an expression for $\bt_n.$
\begin{lem}
The off-diagonal recurrence coefficients $\bt_n$ is expressed in
terms of $\:r_n,$ and $r_n^*$ as follows:
\begin{equation} \label{beta-r-r*}
\{1-(2n+\al+\bt)^2\}\bt_n=-(r_n^*-r_n)^2-(\bt+t)r_n+(t-\al-2n)r_n^*+nt.
\end{equation}
\end{lem}
\begin{proof}
We eliminate $R_n^*$ in favor of $R_n$ using (\ref{r&R}) and replace
$\bt_nR_nR_{n-1}$ by $r_n^2+\bt r_n$ with (\ref{r*&R*}) to find
\begin{equation*}
(r_n^*)^2-tr_n^*=r_n^2+\bt r_n+\bt_n[(2n+\al+\bt)^2-1]
-\bt_n[(2n+1+\al+\bt)R_{n-1}+(2n -1 +\al+\bt)R_{n}].
\end{equation*}
The same substitutions made on (\ref{r&R&R*}) produces
\begin{equation} \label{r-r*-R}
(t-2r_n^*)(n-r_n)-\al r_n^* =2(r_n^2+\bt
r_n)-\bt_n[(2n+1+\al+\bt)R_{n-1}+R_n(2n-1+\al+\bt)].
\end{equation}
Subtracting the above two formulas gives us (\ref{beta-r-r*}).
\end{proof}
\begin{rem}
We consider the case when $t=0.$ In this situation
$R_n^*(0)=r_n^*(0)=0.$ Therefore, from (\ref{r-r2}) and (\ref{Rsum})
we find
$$
R_n(0)=2n + 1 +\al+\bt
$$
and
$$
r_n(0)= \D\frac{n(n+\al)}{2n+\al+\bt},$$ respectively. Finally, from
(\ref{alpha-rs}) and (\ref{beta-r-r*}), we have
\begin{align}
\al_n(0)=&\frac{2n^2+2n(\al+\bt+1)+(1+\al)(\al+\bt)}{(2n+\al+\bt)(2n +\al+\bt+2)},\\
\bt_n(0)=&\frac{n(n+\al)[n^2+(\al+2\bt)n+\bt(\al+\bt)]}{(2n+\al+\bt)^2[(2n+\al+\bt)^2-1]}.
\end{align}
Note that
\begin{align}
\lim_{n\to\infty}\al_n(0)=&\frac{1}{2},\nonumber\\
\lim_{n\to\infty}\bt_n(0)=&\frac{1}{16}.\nonumber
\end{align}
They are in agreement with the classical theory in \cite{Ne}.
\end{rem}

%%%%%%%%%%%%%%%%%%%%%%%%%%%%%%%%%%%%%%%%%%%%%%%%%%%%%%%%%

\section{The $t$ dependance}

Note that our weight function depends on $t$. As a consequence, the
coefficients of our polynomials, the recurrence coefficients and the
auxiliary quantities defined in (\ref{rn1-def})--(\ref{r*n2-def})
all depend on $t$. In this section, we are going to study the
evolution of auxiliary quantities in $t.$ First of all, we state a
lemma which concerns the derivative of $\textsf{p}_1(n)$ with
respect to $t.$

\begin{lem}

We have
\begin{equation} \label{p1n-diff}
t\frac{d}{dt} \textsf{p}_1(n) = r^*_n.
\end{equation}

\end{lem}

\begin{proof}

By the orthogonal property (\ref{orthogonality}), we know
\begin{equation*}
\int_0^1 P_n(x) P_{n-1}(x) w(x;t) dx = 0.
\end{equation*}
Differentiating the above formula with respect to $t$ gives us
\begin{equation*}
\int_0^1 \frac{d}{dt} P_n(x) \; P_{n-1}(x) w(x;t) dx + \int_0^1
P_n(x) P_{n-1}(x) \frac{d}{dt}w(x;t) dx = 0.
\end{equation*}
Using (\ref{w-def}), (\ref{orthogonality}) and (\ref{pn-formula}),
we get
\begin{equation*}
h_{n-1} \frac{d}{dt} \textsf{p}_1(n) - \int_0^1 P_n(x) P_{n-1}(x)
w(x)  \frac{dx}{x} = 0.
\end{equation*}
Taking into account of (\ref{rn2-def}), (\ref{p1n-diff}) follows
immediately.
\end{proof}

From (\ref{p1n-rs}) and the above lemma, it is easily seen that
\begin{equation} \label{p1n-diff-2}
t\frac{d}{dt}\textsf{p}_1(n)= r_n^*
=t\frac{d}{dt}r_n^*-t\frac{d}{dt}r_n
\end{equation}
or
\begin{equation} \label{r&r-diff}
t\frac{d}{dt}r_n^*= r_n^*+t\frac{d}{dt}r_n.
\end{equation}

Next, we have the following property about the Hankel determinant
$D_n$.

\begin{lem}

We have
\begin{equation} \label{dn-R*}
t \frac{d}{dt} \ln D_n(t) = - \sum_{j=0}^{n-1} R^*_j,
\end{equation}
where $R^*_j$ is defined in (\ref{rn1-def}).

\end{lem}

\begin{proof}

Note that the constant $h_n$ defined in (\ref{orthogonality})
depends on the parameter $t$. Then, from (\ref{w-def}) and
(\ref{orthogonality}), we have
\begin{equation}
h_n' = - \int_0^1 [P_n(x)]^2 w(x) \frac{dx}{x}.
\end{equation}
Recalling (\ref{rn1-def}), we get from the above formula
\begin{equation} \label{h'}
h_n' = - \frac{R^*_n h_n}{t},
\end{equation}
which gives us
\begin{equation}
t \frac{d}{dt} \ln h_n = - R^*_n.
\end{equation}
Then, our lemma immediately follows from the above formula and
(\ref{dn-def}).
\end{proof}

From the above lemmas, we also derive differential relations for the
recurrence coefficients $\alpha_n$ and $\beta_n$. These are the
non-standard Toda equations.
\begin{lem}
The recurrence coefficients $\alpha_n$ and $\beta_n$ satisfy the
following differential equations
$$
t\frac{d}{dt} \,\alpha_n  =  r^*_n - r^*_{n+1}, \eqno(T_1) \\
$$
$$
t\frac{d}{dt} \, \beta_n  = (R^*_{n-1} - R^*_n) \, \beta_n,
\eqno(T_2).
$$
where $R^*_n$ and $r^*_n$ are defined in (\ref{rn1-def}) and
(\ref{rn2-def}), respectively.
\end{lem}
\begin{proof}
$(T_1)$ follows from (\ref{alpha-p1n}) and (\ref{p1n-diff}). And
$(T_2)$ follows from (\ref{h'}) and the fact that $\beta_n = h_n /
h_{n-1}$.
\end{proof}

\section{Non-linear differential equation satisfied by $H_n$}

In this section we express $R_n,$ $r_n^*$ and $r_n$ in terms of
$H_n$ and its derivative with respect to $t,$ and obtain a
functional equation involving $H_n,$ $H_n'$ and $H_n''.$ For this
purpose, we first express $r_n$ and $r_n^*$ in terms of $H_n$ and
$H_n'$ in the next Lemma.
\begin{lem}
\begin{align}
r_n^*=&\frac{nt+tH_n'}{2n + \al+\bt}, \label{r*-hn} \\
r_n=&\frac{n(n+\al)+tH_n'-H_n}{2n + \al+\bt}. \label{r-hn}
\end{align}
\begin{proof}
From (\ref{Rsum}) and (\ref{dn-R*}) we have
\begin{equation} \label{hn-r-r*}
-H_n=nt-n(n+\al)-(2n + \al+\bt)(r_n^*-r_n)=nt-n(n+\al)-(2n +
\al+\bt)\textsf{p}_1(n).
\end{equation}
Taking a derivative of the above formula with respect to $t$ and
using (\ref{p1n-diff}), we find
\begin{equation} \label{hn-r*}
-H_n'=n-(2n + \al+\bt)\frac{r_n^*}{t}.
\end{equation}
The equation (\ref{r*-hn}) then follows from the above one. And the
equation (\ref{r-hn}) follows from eliminating $r_n^*$ from
(\ref{hn-r-r*}) and (\ref{hn-r*}).
\end{proof}
\end{lem}

Then we try to get a similar Lemma for $R_n$. To achieve it, we
first derive the relations among $R_n,$ $r_n,$ and $r_n^*.$

\begin{prop}

The auxiliary quantity $R_n(t)$ satisfies the following quadratic
equations
\begin{align}
\frac{2n+1+\al+\bt}{R_n}\:(r_n^2+\bt r_n)+&
\frac{R_n}{2n+1+\al+\bt}\:\left[(r_n^*-r_n)^2+(2n + \al
-t)r_n^*+(\bt+t)r_n-nt\right]
\nonumber\\
=& 2 r_n^2 + (t+ 2 \beta - 2 r_n^*) r_n + (2n + \alpha) r_n^* - nt
\label{R-r-r*}
\end{align}
and
\begin{align}
\frac{1-(2n+\al+\bt)^2 }{R_n}\: (r_n^2+\bt r_n)
+ & \left[(r_n^*-r_n)^2-t(r_n^*-r_n)+ \bt\:r_n+(2n + \al) r_n^*-nt \right]R_n   \nonumber\\
& \hspace{-100pt} = 2 r_n^2 + (t+ 2 \beta - 2 r_n^*) r_n + (2n +
\alpha) r_n^* - nt - (2n + \al+\bt )t\frac{d}{dt}r_n.
\label{R-r-r*2}
\end{align}

\end{prop}

\begin{proof}

In (\ref{r-r*-R}), we replace $\bt_nR_{n-1}$ by
$(r_n^2+\bt\:r_n)/R_n$ with (\ref{r*&R*}) and get
\begin{equation*}
(2n+1+\al+\bt)\frac{r_n^2+\bt\:r_n}{R_n}+(2n-1+\al+\bt)\bt_nR_n
=2(r_n^2+\bt\:r_n)+\al\:r_n^*+(t-2r_n^*)(r_n-n).
\end{equation*}
Replacing $\bt_n$ in the above formula with the aid of
(\ref{beta-r-r*}), we have (\ref{R-r-r*}).

Then, we look back to $(T_2)$ by using (\ref{r-r2}) and
(\ref{r*&R*})
\begin{align}
t\frac{d}{dt}\bt_n=&(R_{n-1}^*-R_n^*)\bt_n\nonumber\\
=&(R_{n-1}-R_n+2)\bt_n\nonumber\\
=&(2-R_n)\bt_n+\frac{r_n^2+\bt\:r_n}{R_n}. \label{beta-diff}
\end{align}
Applying $t\frac{d}{dt}$ to (\ref{beta-r-r*}) gives us
\begin{align*}
[1-(2n+\al+\bt)^2]t\frac{d}{dt}\bt_n =-2(r_n^*-r_n) r_n^*
+t(r_n^*-r_n)+tr_n^*-\bt\:t\frac{d}{dt}r_n-(2n +
\al)t\frac{d}{dt}r_n^* +nt,
\end{align*}
where we have made used of (\ref{p1n-diff-2}) to arrive at the last
step. Substituting (\ref{beta-diff}) into the above formula gives us
(\ref{R-r-r*2}).
\end{proof}

Directly from the above proposition, we express $R_n$ and $1/R_n$ in
terms of $r_n,$ $r_n^*$ and $tr_n'(t).$
\begin{prop}
The auxiliary quantity $R_n$ has the following representations
\begin{align}
R_n(t)=&\frac{(2n+1+\al+\bt)[2r_n^2+(t+2\bt-2r_n^*)r_n+(2n+\al)r_n^*-nt-tr_n'(t)]}
{2 [(r_n^*-r_n)^2+( 2n +  \al - t)r_n^*+(\bt+t)r_n-nt]}, \label{R-r-r*3} \\
\frac{1}{R_n(t)}=&\frac{2r_n^2+ (t+2\bt-2r_n^*)r_n
+(2n+\al)r_n^*-nt+tr_n'(t)} {2 (2n+ 1 +\al+\bt)(\bt+r_n) r_n} .
\label{R-r-r*4}
\end{align}
\end{prop}
\begin{proof}
These are found by solving for $R_n$ and $1/R_n$ from (\ref{R-r-r*})
and (\ref{R-r-r*2}).
\end{proof}

Finally we arrive at the following Theorem.

\begin{thm}
The logarithmic derivative of the Hankel determinant with respect to
$t;$
\begin{equation}
H_n(t):=t\frac{d}{dt}\ln D_n(t),\nonumber
\end{equation}
satisfies the following non-linear second order ordinary
differential equation
\begin{equation} \label{hn-ode}
( t H_n'')^2 = [n (n + \al + \bt) - H_n + (\al + t) H_n']^2 + 4 H_n'
(t H_n' - H_n) (\beta - H_n').
\end{equation}
\end{thm}
\begin{proof}
Multiplying (\ref{R-r-r*3}) and (\ref{R-r-r*4}) gives us
\begin{equation} \label{r-r*}
t^2 [r_n'(t)]^2 = t^2 r_n^2 + 2 r_n[- 2 (2n + \al + \bt) (r_n^*)^2 +
(4n + \al + 2 \bt)t r_n^* - n t^2] + [(2n + \al ) r_n^* - nt]^2.
\end{equation}
Substituting (\ref{r*-hn}) and (\ref{r-hn}) into (\ref{r-r*}), gives
us (\ref{hn-ode}).
\end{proof}

\begin{rem}
It turns out that $H_n$ satisfies the Jimbo-Miwa-Okamoto $\sigma$
form of $P_V$ for a special choice for parameters.

If
\begin{equation} \label{htil-def}
\widetilde{H}_n:=H_n-n(n+\al+\bt),
\end{equation}
then (\ref{hn-ode}) becomes
\begin{align}
(t
\widetilde{H}_n'')^2=&-4t(\widetilde{H}_n')^3+(\widetilde{H}_n')^2
\biggl[ 4\widetilde{H}_n +
(\al+2\bt+t)^2+4n(n+\al+\bt) -4\bt(\al+\bt) \biggr]\nonumber\\
& + 2 \widetilde{H}_n' \biggl[- (\al+2\bt+t)\widetilde{H}_n - 2 n
\bt(n+\al+\bt) \biggr]+\widetilde{H}_n^2.\nonumber
\end{align}
Comparing the above formula with the Jimbo-Miwa-Okamoto
\cite{JM,okamoto} $\sigma-$form of $P_V$, we can choose a possible
identification of the parameters of \cite{JM} to be
\begin{equation}
\nu_0=0,\;\nu_1=-(n+\al+\bt),\;\nu_2=n,\;\nu_3=-\bt.
\end{equation}
\end{rem}

When suitably limit is taken, (\ref{hn-ode}) can be reduced to the
$\sigma-$form of a Painlev\'e III in \cite{ci} . To see this, we
replace $x$ by $y/\bt,$ and $t$ by $s/\beta,$ in
$$
\int_{0}^{1} \: e^{-t/x}x^{\al}(1-x)^{\bt}P_n^2(x)dx=h_n(t),
$$
resulting
\begin{equation}\label{reduce}
\int_{0}^{\bt} \: e^{-s/y} y^{\al}
\left(1-\frac{y}{\bt}\right)^{\bt}\widetilde{P}_n^2(y)dy
=\bt^{2n+\al+1}\:h_n(s/\bt),
\end{equation}
where $\widetilde{P}_n(y)=\bt^n\:P_n(y/\bt).$ Now since the left
hand side of (\ref{reduce}) tends to
$$
\int_{0}^{\infty}\: e^{-s/y-y}y^{\al}\widetilde{P}_n^2(y)dy,
$$
as $\bt\to\infty,$ we see that
$$
\lim_{\bt\to\infty}\bt^{2n+1+\al}h_n(s/\bt)
$$
becomes the square of the $L^2$ norm of the orthogonal polynomials
studies in \cite{ci}. Consequently,
$$
\lim_{\bt\to\infty}\bt^{n(n+\al)}D_n(s/\bt)
$$
becomes the Hankel determinant
$$
\det\left(\int_{0}^{\infty}y^{i+j}\:{\rm
e}^{-s/y-y}y^{\al}dy\right)_{i,j=0}^{n-1}.
$$
Indeed replace $t$ by $s/\bt$ and let $\bt\to\infty,$ (\ref{hn-ode})
becomes, keeping only the highest order term in $\bt,$
\begin{equation}\label{scale}
(sH_n'')^2=(n+\al H_n')^2+4(sH_n'-H_n)H_n'(1-H_n'),
\end{equation}
which is (3.24) of \cite{ci} in a slightly different form.

Note that we have abused the notation: retaining $H_n$ after the
limit to avoid introducing extra symbols.

%%%%%%%%%%%%%%%%%%%%%%%%%%%%%%%%%%%%%%%%%%%%%%%%%%%%%%%%%%%%%%%%%%%%%%%%%%%%%%%%%%

\section{Non-linear difference equation satisfied by $H_n$}

Based on recent papers \cite{basor-chen, ci}, we expect to find a
second order non-linear difference equation satisfied by $H_n.$ To
arrive at the difference equation, we want to express the recurrence
coefficients $\al_n$ and $\bt_n$ in terms of $H_n$ and $H_{n \pm
1}.$ First, let us find out the useful relation between $H_n$ and
$\textsf{p}_1(n)$.

\begin{lem}
We have
\begin{equation} \label{hn-p1n}
H_n = (2n + \alpha + \beta ) \textsf{p}_1(n) + n ( n + \alpha - t)
\end{equation}
or
\begin{equation}
\textsf{p}_1(n) = \frac{H_n - n ( n + \alpha - t)}{2n + \alpha +
\beta}.
\end{equation}

\begin{proof}
From (\ref{Rsum}), (\ref{dn-R*}) and the definition of $H_n$, we
have
\begin{align}
H_n & = t \frac{d}{dt} \ln D_n(t) = - \sum_{j=0}^{n-1} R^*_j \label{hn-sumR} \\
& = (2n + \alpha + \beta ) (r^*_n - r_n ) + n ( n + \alpha - t).
\end{align}
Using (\ref{p1n-rs}), rewriting the above formula gives our
proposition.
\end{proof}

\end{lem}

Note that since $\al_n=\textsf{p}_1(n)-\textsf{p}_1(n+1),$ we have a
very simple expression for $\al_n$ in terms of $H_{n}$ and
$H_{n+1}.$ Furthermore we can also obtain $R_n$ in terms of $H_n$ or
$\textsf{p}_1(n)$.
\begin{lem}
\begin{align}
R_n & = H_n - H_{n+1} + 2n + 1 + \al + \bt \\
& = (2n + \alpha + \beta ) \textsf{p}_1(n) - (2n + 2 + \alpha +
\beta) \textsf{p}_1(n+1) + t + \bt. \label{Rn-p1n}
\end{align}
\end{lem}
\begin{proof}
These are found by combining (\ref{r-r2}), (\ref{hn-sumR}) and
(\ref{hn-p1n}) together.
\end{proof}
For the formulas involving $r_n$ and $\bt_n$, we have the results as
follows.
\begin{lem}
We have
\begin{equation} \label{rn-p1n}
r_n = \frac{- [\textsf{p}_1(n)]^2 - (2n + \al - t) \textsf{p}_1(n) +
nt - [1 - (2n + \al + \bt)^2] \beta_n}{2n + \al + \bt}
\end{equation}
with
\begin{equation} \label{beta-p1n}
\bt_n = \frac{X_n}{Y_n},
\end{equation}
where
\begin{equation}
X_n := \frac{1}{2n + \al + \bt } \biggl[ - 2 [\textsf{p}_1(n)]^3 +
(3t - \al + 2 \bt -2 n) [\textsf{p}_1(n)]^2 - (t^2 - 2 (n - \bt) t -
(2n + \al ) \bt) \textsf{p}_1(n) - (t + \bt) nt \biggr]
\end{equation}
and
\begin{align}
Y_n & := (2n -1 + \al + \bt) (2n + 1 + \al + \bt) + \frac{2}{2n + \al + \bt} \textsf{p}_1(n) \nonumber \\
 & + (2n -1 + \al + \bt) (2n + 2 + \al + \bt)\textsf{p}_1(n+1) - (2n + 1 + \al + \bt) (2n - 2 + \al + \bt)\textsf{p}_1(n-1)
 \nonumber \\
 & - (t +
\bt) \left( \frac{1}{2n + \al + \bt} + 2n + \al + \bt \right).
\end{align}
\end{lem}
\begin{proof}
To obtain (\ref{rn-p1n}) and (\ref{beta-p1n}), first we use
(\ref{r&R}) and (\ref{p1n-rs}) to get
\begin{equation}
r_n^2 + (2 \, \textsf{p}_1(n) - t) r_n + [\textsf{p}_1(n)]^2 - t
\textsf{p}_1(n) = \beta_n  R^*_n R^*_{n-1}.
\end{equation}
Subtracting the above formula from (\ref{r*&R*}), we have
\begin{equation}
(2 \, \textsf{p}_1(n) - t - \bt) r_n + [\textsf{p}_1(n)]^2 - t
\textsf{p}_1(n) = \beta_n ( R^*_n R^*_{n-1} - R_n R_{n-1} ),
\end{equation}
which is a linear equation with respect to $r_n$ and $\bt_n$. Recall
that (\ref{beta-r-r*}) is also linear in $r_n$ and $\bt_n$. Solving
this linear system and taking into account (\ref{r-r2}) and
(\ref{Rn-p1n}), we prove our lemma.
\end{proof}

Finally, we obtain the following theorem.

\begin{thm}

$\widetilde{H}_n$ satisfies the following non-linear second order
ordinary difference equation
\begin{align}
nt & - \frac{(2n + \al ) ( \widetilde{H}_n + n (\bt + t))}{2n + \al
+ \bt} + \frac{\widetilde{Z}_n}{2n + \al + \bt} \biggl[ -2n - \al -
t +
\frac{2 (\widetilde{H}_n + n (\bt + t))}{2n + \al + \bt} \biggr] \nonumber \\
 + & \frac{2 \widetilde{Z}_n^2}{(2n + \al + \bt)^2} = \frac{1}{Z_n} \biggl( -2 \widetilde{H}_n^2 + 2 \widetilde{H}_n (1 + \widetilde{H}_{n+1})
 + \widetilde{H}_{n+1} (2n - 1 +\al +\bt) \nonumber \\
 & \hspace{120pt} -  \widetilde{H}_{n-1} (2n + 1 + \al + \bt - 2 \widetilde{H}_n + 2 \widetilde{H}_{n+1}) \biggr) \nonumber \\
 & \hspace{40pt} \times \biggl( -  n t (\bt + t) + \frac{(\al \bt + 2 \bt (n - t) + (2 n - t) t) (\widetilde{H}_n + n (\bt + t))}{2n + \al + \bt}&
 \nonumber \\
 & \hspace{50pt} +  \frac{(-\al + 2 \bt - 2 n + 3 t) (
\widetilde{H}_n + n (\bt + t))^2}{(2n + \al + \bt)^2} - \frac{2(
\widetilde{H}_n + n (\bt + t))^3}{(2n + \al + \bt)^3}  \biggr),
\label{final-difference}
\end{align}
where
\begin{align*}
Z_n:= & (2n + \al + \bt) \biggl( (2n -1 + \al + \bt ) (2n + 1 + \al + \bt) -(2n + \al + \bt + \frac{1}{2n + \al + \bt})(\bt + t)  \\
 & - (2n + 1 + \al + \bt) (\widetilde{H}_{n-1} + (n-1 ) (\bt + t)) + \frac{2 ( \widetilde{H}_{n} + n (\bt + t))}{(2n + \al + \bt)^2} \\
 & + (2n-1 + \al + \bt) ( \widetilde{H}_{n+1} + (n+1) (\bt + t)) \biggr), \\
\widetilde{Z}_{n} : = & n t + \frac{(-2n -\al + t)(\widetilde{H}_n +
n (\bt + t))}{2n + \al + \bt} - \frac{(\widetilde{H}_n + n (\bt +
t))^2}{(2n + \al + \bt)^2} - \frac{1 - (2n + \al +
\bt )^2}{Z_n}\biggl(-n t (\bt + t)  \\
& + \frac{(\al \bt + 2 \bt (n - t) + (2 n - t) t)(\widetilde{H}_n +
n (\bt + t))}{2n + \al + \bt} + \frac{(-\al + 2 \bt - 2 n + 3
t)(\widetilde{H}_n + n (\bt +
t) )^2}{(2n + \al + \bt)^2} \\
& - \frac{2(\widetilde{H}_n + n (\bt + t) )^3}{(2n + \al + \bt)^3}
\biggr).
\end{align*}

\end{thm}
\begin{proof}

The non-linear difference equation for $\textsf{p}_1 (n)$ is
obtained by substituting (\ref{r-r2}), (\ref{p1n-rs}) and
(\ref{Rn-p1n})--(\ref{beta-p1n}) into (\ref{r&R&R*}). Due to the
relation among $\textsf{p}_1 (n)$, $H_n$ and $\widetilde{H}_n$ in
(\ref{htil-def}) and (\ref{hn-p1n}), the non-linear difference
equation for $\widetilde{H}_n$ follows.
\end{proof}

%%%%%%%%%%%%%%%%%%%%%%%%%%%%%%%%%%%%%%%%%%%%%%%%%%%%%%%%%%%%%%%%

\section{$P_V((2n+1+\al+\bt)^2/2,-\bt^2/2,\al,-1/2)$}

We end this paper with the derivation of a second order ordinary
differential equation for $R_n$ which is expected to a $P_V$ since
we have seen that $H_n$ satisfies the Jimbo-Miwa-Okamoto
$\sigma-$form.

For this purpose we state in the next Lemma a Riccati equation
satisfied by $R_n$.
\begin{lem}
The auxiliary quantity $R_n(t)$ satisfies the following Riccati
equation,
\begin{equation}\label{Rn-ode1}
t R_n' = 2 R_n (r^*_n - r_n) + (2n + 1 + \al + \bt) (2 r_n -R_n +
\bt) + (R_n - \bt - t)R_n.
\end{equation}
\end{lem}
\begin{proof}
First we apply $t\frac{d}{dt}$ to equation (\ref{alpha-rs}) and make
use of $(T_1)$ to replace $t\frac{d }{dt} \al_n$ by
$r_n^*-r_{n+1}^*.$

In the next step we replace $r_{n+1}^*$ by $t-\al_nR_n^*-r_n$ using
(\ref{r-r1}). Finally, noting (\ref{alpha-rs}) and
$R_n^*=R_n-(2n+1+\al+bt),$ we arrive at (\ref{Rn-ode1}).
\end{proof}
From (\ref{Rn-ode1}) we see that
\begin{equation} \label{rn*-R}
r^*_n = \frac{1}{2 R_n} \biggl[ t R_n' - (2n + 1 + \al + \bt) (2 r_n
- R_n + \bt )\biggr] + r_n - \frac{R_n - \bt - t}{2}.
\end{equation}
Substituting the above formula into (\ref{R-r-r*3}) and
(\ref{R-r-r*4}), we find a pair of linear equations in $r_n$ and
$r_n'$. Solving this system we have
\begin{align}
r_n  = &F(R_n, R_n'), \label{rn-rn'1}\\
r_n' = &G(R_n, R_n'), \label{rn-rn'2}
\end{align}
where $F(\cdot, \cdot)$ and $G(\cdot, \cdot)$ are functions that are
explicitly known. Because the expressions are unwieldy, we have
decided not to write them down.

By equating the derivative of (\ref{rn-rn'1}) with respect to $t$
and (\ref{rn-rn'2}), we find, after some simplification the
following:
\begin{align}
&[(2n+\al+\bt)(2n+1+\al+\bt)-(4n+2\al+2\bt+1)R_n+R_n^2-tR_n']\nonumber\\
& \times \biggl[2t^2(2n+1+\al+\bt-R_n)R_n\:R_n''+t^2(3R_n-2n-1-\al-\bt)(R_n')^2\nonumber\\
&+2t(2n+1+\al+\bt-R_n)R_n\:R_n'+R_n^5-3(2n+1+\al+\bt)R_n^4\nonumber\\
&+C_1(t)R_n^3+C_2(t)R_n^2
-3\bt^2(2n+1+\al+\bt)^2\:R_n+\bt^2(2n+1+\al+\bt)^3\biggr]=0,
\label{two-ode}
\end{align}
where
\begin{align}
C_1(t): =& -t^2 + 2 \al t + 3 (2n+1+\al+\bt)^2 - \bt^2 \nonumber\\
C_2(t): =& -(2n+1+\al+\bt) \biggl[  t^2 + 2 \al t + (2n+1+\al+\bt)^2
- 3 \bt^2 \biggr].\nonumber
\end{align}
From (\ref{two-ode}) we have two equations, one of which is a
Riccati equation whose solution is
\begin{equation}
R_n(t)=\frac{(2n+\al+\bt)\gamma_n t-(2n+1+\al+\bt)}{\gamma_n t-1},
\end{equation}
where $\gamma_n$ is an integration constant. Note that $R_n(t)$
tends to $2n+\al+\bt$ as $t\to\infty.$ Because $R_0(t)\sim t$ from
Remark (\ref{R0-asy}), we discard this equation.

It turns out that the above differential equation for $R_n(t)$ is a
Painlev\'e V.
\begin{thm}
Let
\begin{equation}
S_n(t) := \frac{R_n(t)}{2n+1+\al+\bt}.
\end{equation}
Then $S_n(t)$ satisfies the following differential equation
\begin{align}
S_n'' = & \; \frac{3 S_n - 1}{2 S_n (S_n - 1)} (S_n')^2 -
\frac{S_n'}{t}
+ \frac{(S_n-1)^2}{t^2} \left[\frac{(2n+1+\al+\bt)^2}{2}\:S_n-\frac{\bt^2}{2S_n}\right] \nonumber \\
& + \frac{\al S_n}{t} - \frac{S_n (S_n + 1) }{2 (S_n - 1)},
\label{w-ode}
\end{align}
which is $P_V((2n+1+\al+\bt)^2/2,-\frac{\bt^2}{2},\al,-1/2).$
\end{thm}
\begin{proof}
The equation (\ref{w-ode}) follows if we substitute
$$
R_n(t)=(2n+1+\al+\bt)S_n(t)
$$
into the second order ODE implied by (\ref{two-ode}).
\end{proof}

\medskip
\noindent \textbf{Acknowledgement.} \vskip .3cm \noindent Yang Chen
is partially supported by EPSRC Grant $\sharp$ R27027. He would also
like to thank the School of Mathematical Sciences, Fudan University,
Shanghai, China, for their hospitality where this paper was
completed. Dan Dai is partially supported by K.U.Leuven research
grants OT/04/21 and OT/08/33, and by the Belgian Interuniversity
Attraction Pole P06/02. He would like to thank the ENIGMA programme
for support of his visit to the Department of Mathematics, Imperial
College London.

%%%%%%%%%%%%%%%%%%%%%%%%%%%%%%%%%%%%%%%%%%%%%%%%%%%%%%%%%%%%%%%%%%%%%%%%%%%

\end{document}